\documentclass[11pt]{amsart}

\usepackage{amsfonts, amsmath, amssymb, amsgen, amsthm, amscd}

\usepackage[utf8]{inputenc} 
\usepackage[top=3cm, left=2.5cm, bottom=2.5cm, right=2.5cm]{geometry}

\usepackage{enumerate}
\usepackage{color}
\usepackage[all,cmtip]{xy}

\def\a{\alpha}
\def\b{\beta}
\def\d{\delta}
\def\D{\Delta}

\def\t{\theta}

\def\ve{\varepsilon}

\def\Id{\mathop{\rm Id}\nolimits}

\def\lra{\longrightarrow}

\def\ot{\otimes}

\def\lra{\longrightarrow}

\def\D{\Delta}

\def\Id{\mathop{\rm Id}\nolimits}

\newcommand{\ns}[1]{~\hspace{-4pt}_{_{{<#1>}}}}

\newcommand{\C}[1]{\mathcal{#1}}
\newcommand{\B}[1]{\mathbb{#1}}

\newcommand{\wbar}[1]{\overline{#1}}

\newcommand{\CB}{{\rm CB}}

\newcommand{\Hom}{{\rm Hom}}

\newcommand{\cotor}{{\rm Cotor}}

\newcommand{\ie}{{\it i.e.\/}\ }

\newcommand{\Div}{\C{D}\text{iv}}
\newcommand{\Bin}{\C{B}\text{in}}
\newcommand{\Dir}{\C{D}\text{ir}}

\renewcommand{\leq}{\leqslant}
\renewcommand{\geq}{\geqslant}

\numberwithin{equation}{section}

\newtheorem{theorem}{Theorem}[section]
\newtheorem{proposition}[theorem]{Proposition}
\newtheorem{lemma}[theorem]{Lemma}
\newtheorem{corollary}[theorem]{Corollary}
\theoremstyle{definition}

\newtheorem{remark}[theorem]{Remark}
\newtheorem{example}[theorem]{Example}
\newtheorem{definition}[theorem]{Definition}

\title{Hochschild Cohomology of Reduced Incidence Algebras}

\author{M. KANUN\.I}
\author{A. Kaygun}
\author{S. Sütlü}

\begin{document}

\maketitle

\begin{abstract}
We compute the Hochschild cohomology of the reduced incidence algebras such as the algebra of
formal power series, the algebra of exponential power series, the algebra of Eulerian power series, and the
algebra of formal Dirichlet series. We achieve the result by carrying out the computation on the coalgebra $\cotor$-groups of their pre-dual coalgebras. Using the same coalgebraic machinery, we further identify the Hochschild cohomology groups of an incidence algebra associated to a quiver with the ${\rm Ext}$-groups of the incidence algebra associated to a suspension of the quiver.
\end{abstract}

\section*{Introduction}

Employing coalgebraic methods in performing otherwise difficult
calculations is hardly new.  G.R. Rota was a proponent, and an
effective user of this approach in combinatorics \cite{Rota95}.  We would like, now,
to use a similar strategy in cohomology.  

Unfortunately, the fact that vector space duality is not an
equivalence on any category of (infinite dimensional) vector spaces is
a hindrance in viewing algebraic and coalgebraic methods on the same
footing.  For example, the vector space dual of every coalgebra is an
algebra, but not every algebra has a pre-dual coalgebra.  When this
happens, one can use coalgebraic methods calculating invariants of
such algebras with pre-duals.

In this paper we demonstrate, by concrete calculations, that
coalgebraic methods bring in a new set of tools in calculating
cohomological invariants of a large class of associative algebras.

In~\cite{Cibi89}, Cibils gives a functorial construction on quivers which can be viewed 
as a \emph{suspension functor}. He shows that the simplicial cohomology groups of 
the simplicial complex associated to a quiver, and the ${\rm Ext}$-groups of the incidence algebra 
associated to the suspension are isomorphic with a degree 2
shift. Combined with the result of Gerstenhaber and Schack in \cite{GersScha83}, identifying the 
simplicial cohomology of the simplicial complex associated to a quiver with the Hochschild 
cohomology of the incidence algebra associated to it, one obtains an isomorphism between the 
Hochschild cohomology groups of a quiver incidence algebra, and ${\rm Ext}$-groups of the incidence algebra 
associated to the suspension, with a degree 2 shift.
We obtain the combination of these two results directly without referring to 
the simplicial cohomology, using our coalgebraic machinery, in Subsection \ref{Suspension}.

The following result of Subsection~\ref{Calculations} is new. We calculate the continuous Hochschild cohomology of four
associative algebras: the algebra of
formal power series, the algebra of exponential power series, the algebra of Eulerian power series, and the
algebra of formal Dirichlet series where the first three are
isomorphic as algebras.  These algebras are all examples of incidence
algebras.  We obtain the isomorphisms by furnishing their pre-dual
coalgebras (called \emph{incidence coalgebras}) and then
proving these incidence coalgebras are isomorphic in
Section~\ref{Isomorphisms}.  Then we verify that the cohomology of
these coalgebras carry a ring structure, and we explicitly calculate
the Hochschild cohomology rings of their duals.

The plan of the paper is as follows. We begin with a review of incidence
algebras.  Then in the next section we introduce incidence coalgebras.
In the third section we construct the isomorphisms we mentioned above.
In the following section we introduce the cohomological background
needed to perform our cohomological calculations, and in the final
section we make our calculations.

\subsection*{Notations and Conventions}
Throughout the text all (co)algebras are (co)unital and (co)associative, over a base field $k$ of characteristic 0. 
We make no assumptions that (co)algebras are finite dimensional, or (co)commutative. All unadorned tensor
products are assumed to be over $k$. The cotensor product is denoted by $\Box$. 

For a finite set $U$, we use $|U|$ to denote the cardinality of
$U$. A {\it graph} $G=(V,E)$ has a vertex set $V$, and an edge set $E$ 
\begin{align*}
E \subset \left(\begin{array}{c} V \\ 2 \end{array} \right).
\end{align*}
If $|V|$ is finite, then $G$ is called a {\it finite graph}. A {\it
  tree} is a graph in which, any two vertices has a unique path
between them. A {\it forest} is a disjoint union of trees.  A {\it quiver} $Q=(Q_0,Q_1)$ is a directed graph with
the vertex set $Q_0$ and the arrow set $Q_1$, with initial and
terminal points of each arrow being in $Q_0$. An {\it oriented path}
is a sequence of arrows concatenated by respecting the orientation of
each arrow. A {\it cycle} is a path whose initial and terminal vertex
is the same.
A vertex is called a {\it source} if it is not the terminal vertex of any arrow, and a vertex is called a {\it sink} if it is 
not the initial vertex of any arrow.  

\section{Incidence Algebras}

In this section we recall the notion of an (reduced) incidence algebra from \cite{Rota64,DoubRotaStan72,SpiegelODonnell-book}. We also discuss briefly the duality between an incidence Hopf algebra, and a reduced incidence algebra. This duality will be essential in computing the (co)homology of a reduced incidence algebra.

\subsection{Basic definitions}

Let us recall that a partially ordered set, poset in short, $(P ,\leq)$ is called {\it locally finite} if, given any $x,y\in P$, the interval 
\begin{equation*}
[x,y] = \{z\in P |\ x\leq z\leq y \} 
\end{equation*}
is finite.

\begin{example}
  The set of natural numbers $\B{N}$, or the set of integers $\B{Z}$ are both locally finite partially ordered sets with respect to
  the standard order relation $\leq$.
\end{example}

Let $(P,\leq)$ be a locally finite partially ordered set. Then
\begin{equation*}
{\bf I}(P) := \{f:P\times P \to k\mid f(x,y)=0,\text{ if } x \nleq y\}
\end{equation*}
becomes an algebra, called the {\it incidence algebra of $P$}, with the multiplication
\begin{equation*}
(fg)(x,y) = \sum_{x\leq z\leq y} f(x,z)g(z,y).
\end{equation*}
We note that the identity element of ${\bf I}(P)$ is given by
\begin{equation*}
\delta(x,y) = \begin{cases}
1, \text{ if } x = y, \\
0, \text{ if } x \neq y.
\end{cases}
\end{equation*}

Let us note also from \cite[Prop. 3.1]{DoubRotaStan72} that the incidence algebra ${\bf I}(P)$ is actually a topological algebra, equipped with the standard topology which is given as follows: a sequence $\{f_n\}_{n\in \B{N}}$ converges to $f \in {\bf I}(P)$ if and only if $\{f_n(x,y)\}_{n\in \B{N}}$ converges to $f(x,y)$ in $k$ for any $x,y \in P$.

In order to define a reduced incidence algebra, we recall the order compatibility on the segments of a partially ordered set from \cite[Def. 4.1]{DoubRotaStan72}, see also \cite[Sect. 3]{Schm94}.

\begin{definition}
Let $P$ be a locally finite partially ordered set, and $\sim$ an equivalence relation on the intervals of $P$. The relation $\sim$ is called {\it order compatible} if for any two $f,g \in {\bf I}(P)$, $f(x,y)=f(u,v)$ and $g(x,y)=g(u,v)$ for any pair of intervals $[x,y]\sim[u,v]$ implies that $(fg)(x,y) = (fg)(u,v)$.
\end{definition}

Given an order compatible equivalence relation $\sim$, let $\alpha := \wbar{[x,y]}$ denote the equivalence class of an interval $[x,y]$ with respect to $\sim$, called the {\it type of $[x,y]$}. Then, the set ${\bf R}(P,\sim)$ of all functions defined on the types is an algebra with the multiplication given by
\begin{equation*}
(fg)(\alpha) = \sum\left[\begin{array}{c}
\alpha \\
\beta,\,\gamma
\end{array}\right] f(\beta)g(\gamma),
\end{equation*}
where 
\begin{align*}
\left[\begin{array}{c}
\alpha \\
\beta,\,\gamma
\end{array}\right] &:= \begin{cases}
\text{ the number of distinct elements $z$ in an interval $[x,y]$ with $\a=\wbar{[x,y]}$ such that } \\
 \text{ $\wbar{[x,z]} =\b$, and $\wbar{[z,y]} =\gamma$.}
\end{cases}
\end{align*}
The algebra $ {\bf R}(P,\sim)$ is called the {\it reduced incidence algebra}. It is noted in \cite[Prop. 4.3]{DoubRotaStan72} that ${\bf R}(P,\sim)$ is a subalgebra of ${\bf I}(P)$, where the inclusion $\hat\,:{\bf R}(P,\sim) \lra {\bf I}(P,k)$, $f \mapsto \hat{f}$ is given by $\hat{f}(x,y):=f(\alpha)$ if the segment $[x,y]$ is of type $\alpha$.

We proceed to the examples of incidence algebras.

\subsection{Formal power series}\label{ex1} \cite[Ex. 4.5]{DoubRotaStan72}.
Let $\B{Z}^{\geq 0} = \{0,1,2,\ldots \}$ be the set of non-negative integers with
the usual ordering. Let also $\sim$ be the isomorphism of
intervals. Then the type $\wbar{[i,j]}$ of an interval $[i,j]$ is
determined by the non-negative number $j-i \geq 0$. Furthermore,
\begin{equation*}
(f\cdot g)\left(\wbar{[i,j]}\right) = \sum_{i\leq k \leq j} f\left(\wbar{[i,k]}\right)g\left(\wbar{[k,j]}\right), 
\end{equation*}
that is, for $n:=j-i$, and $r:=k-i$,
\begin{equation*}
(f\cdot g)(n) = \sum_{r=0}^n f(r)g(n-r). 
\end{equation*}
As a result,
\begin{equation*}
{\bf R}(\B{Z}^{\geq 0},\sim) \lra k[[x]], \qquad f \mapsto \sum_{n=0}^\infty f(n)x^n
\end{equation*}
is an isomorphism of algebras.

\subsection{Exponential power series}\label{ex2} \cite[Ex. 4.6]{DoubRotaStan72}. Let $B(S)$ be the family of all finite subsets of a countable set $S$, ordered by inclusion. Let, as earlier, $\sim$ be the isomorphism of intervals. Then the type $\wbar{[A,B]}$ is determined by the number $|B\backslash A|$ of the elements of the set $B\backslash A$, see also \cite[Def. 3.2.17]{SpiegelODonnell-book}. Hence, an element $f \in {\bf R}(B(S),\sim)$ gives rise to a function on the set of non-negative integers, $f(n):=f(A,B)$ for any $[A,B] \subseteq B(S)$ with $|B\backslash A| = n$, and
\begin{equation*}
f \mapsto \sum_{n=0}^\infty \frac{f(n)}{n\,!}x^n
\end{equation*}
identifies ${\bf R}(P,\sim)$ with the algebra of exponential power series. Under this identification, the multiplication corresponds to the binomial product
\begin{equation*}
(f\cdot g)(n) = \sum_{r=0}^n \binom{n}{k} f(r)g(n-r).
\end{equation*} 

\subsection{Eulerian power series}\label{ex4} \cite[Ex. 4.9]{DoubRotaStan72}. Let $V$ be a vector space of countable dimension on ${\rm GF}(q)$, and $L(V)$ the lattice of finite dimensional subspaces of $V$. Let $\sim$ be the equivalence relation defined by $[X,Y] \sim [U,W]$ if $Y/X \cong W/U$, that is, $\dim(Y)-\dim(X) = \dim(W)-\dim(U)$. Therefore, the type of an interval corresponds to an integer. Furthermore, since the number of $j$-dimensional subspaces of an $n$-dimensional vector space over the finite field ${\rm GF}(q)$ is the q-binomial coefficient
\begin{align*}
\left[\begin{array}{c}
n \\
j 
\end{array}\right]_q = \frac{[n]_q!}{[n-k]_q! [k]_q!},
\end{align*}
see \cite[Thm. 7.1]{KlimSchm-book,KacCheung-book}, see also \cite[Ex. 2.2.21]{SpiegelODonnell-book}, where 
\begin{equation*}
  [n]_q! := [1]_q [2]_q\cdots [n]_q = \frac{(q-1)(q^2-1)\cdots(q^n-1)}{(q-1)^n},
\end{equation*}
and the quantum integers $[n]_q$ are defined as
\begin{equation*}
  [n]_q = \frac{q^n-1}{q-1} = 1 + q + \cdots + q^{n-1},
\end{equation*}
the multiplication on ${\bf R}(L(V),\sim)$ is given by
\begin{equation}\label{Eulerian}
(f\cdot g)(n) 
= \sum_{r=0}^n {n\brack r}_q f(r)g(n-r) = \sum_{r=0}^n \frac{(q^n-1)(q^n-q)\ldots (q^n-q^{r-1})}{(q^r-1)(q^r-q)\ldots (q^r-q^{r-1})}f(r)g(n-r).
\end{equation}
As a result, the correspondence 
\begin{equation*}
f \mapsto \sum_{n\geq 0} \frac{f(n)}{[n]_q!} (q-1)^n x^n
\end{equation*}
identifies ${\bf R}(L(V),\sim)$ with the algebra of Eulerian power series. 

\subsection{Formal Dirichlet series}\label{ex3} \cite[Ex. 4.7]{DoubRotaStan72}. Let $\B{Q}$ be the additive group of rational numbers modulo $1$, and $L(\B{Q})$ the lattice of subgroups excluding $\B{Q}$. Since every proper subgroup of $\B{Q}$ is finite cyclic, setting $[X,Y] \sim [U,V]$ in $L(\B{Q})$ if $Y/X \cong V/U$, we may identify the type of an interval with a positive integer. Therefore, the multiplication in the reduced incidence algebra ${\bf R}(L(\B{Q}),\sim)$ is given by
\begin{equation*}
(f\cdot g)(n) = \sum_{ij=n} f(i)g(j),
\end{equation*}
and hence
\begin{equation*}
f \mapsto \sum_{n\geq 1} \frac{f(n)}{n^s}
\end{equation*}
identifies ${\bf R}(L(\B{Q}),\sim)$ with the algebra of formal Dirichlet series. 

An alternative construction that reflects the terminology is given in \cite[Sec.3 Ex.1]{Rota64}. Let the set of natural numbers $\mathbb{N}$ be partially ordered by the divisibility, and let $[x,y] \sim [u,v]$ in $P$ if 
$y/x = v/u$. As a result, the reduced incidence algebra ${\bf R}(P,\sim)$ is again the algebra of formal Dirichlet series with the above multiplication and identification. 

\subsection{Incidence algebras of quivers}\cite{Cibi89,GatiRedo01}.
Let $Q=(Q_0,Q_1)$ be an {\em ordered quiver} (a finite quiver without oriented cycles, such that for each arrow $x\overset{\a}{\lra}y\in Q_1$ there is no oriented path joining $x$ to $y$ other than $\a$). In this case, the set $Q_0$ of vertices of $Q$ becomes a finite poset via $x\geq y$ if there is an oriented path from $x\in Q_0$ to $y\in Q_0$. 

Let $kQ$ be the path algebra of a quiver $Q$, and $I$ the two-sided ideal of $kQ$ generated by the differences of paths with the same initial and terminal points. Then the algebra $\C{A}:=kQ/I$ is the incidence algebra of the poset associated to the quiver $Q$, \ie $\C{A}={\bf I}(Q_0)$. 

\subsection{Commutative graph incidence (Hopf) algebras}\label{graph-incidence} \cite[Sect. 12]{Schm94}. 
Let $G=(V,E)$ be a finite graph. Given a subset $U\subseteq V$, the induced subgraph $G\mid U$ is defined to be the graph with the vertex set $U$, and the edge set of all edges of $G$ with both end-vertices in $U$.

Let $\C{G}$ be the family of graphs which is closed under the formation of induced graphs, and sums, and $\widetilde{\C{G}}$ the set of isomorphism classes of graphs in $\C{G}$.

Given $G\in \C{G}$, let $B(G)$ be the lattice of subsets of $V$, ordered by inclusion, and $\C{P}(\C{G})$ the (hereditary, see \cite[Sect. 4]{Schm94}) family of all finite products of intervals from the posets of $B(G)$, for $G \in \C{G}$. For 
$G_i=(V_i,E_i),G_j'=(V_j',E_j') \in \C{G}$, $U_i \subseteq W_i \subseteq V_i$, $U'_j \subseteq W'_j \subseteq V'_j$, $1\leq i \leq n$, $1\leq j \leq m$, a (Hopf) relation on $\C{P}(\C{G})$ is given by 
\begin{eqnarray}\notag
& [U_1,W_1]\times \ldots \times [U_n,W_n] \sim [U'_1,W'_1]\times \ldots \times [U'_m,W'_m] \\\label{sim}
& \text{if and only if}\\\notag
&G_1 \mid (W_1-U_1) + \ldots + G_n \mid (W_n-U_n) \sim G'_1 \mid (W'_1-U'_1) + \ldots + G'_m \mid (W'_m-U'_m) \,\,\text{as graphs},
\end{eqnarray}
where the sum is defined to be the disjoint union of the graphs. The set $\widetilde{\C{P}(\C{G})}$ of isomorphism classes of (finite products of) intervals can then be identified with $\widetilde{\C{G}}$ via
\begin{equation*}
[U,W] \mapsto [G\mid (W-U)],
\end{equation*}
where the bracket on the right stands for the isomorphism class of the graph. The (reduced) incidence algebra ${\bf I}(\C{P}(\C{G}),\sim)$ is (identified with, see \cite[Sect. 7]{Schm94}) the set of all maps $\widetilde{\C{P}(\C{G})} \to k$ with the (convolution) multiplication
\begin{equation*}
fg([G]) := \sum_{U\subseteq V}f([G\mid U])g([G\mid (V-U)]).
\end{equation*}

\section{Incidence coalgebras}\label{sect-IncCoalg}

We now recall incidence coalgebras from \cite{Schm94,Schm87}. Let $\C{P}$ be a family of locally finite partially ordered sets, which is interval closed. That is, $\C{P}$ is non-empty, and if $x \leq y$ in $\C{P}$ then the interval $[x,y]$ belongs to $\C{P}$. Let also $\sim$ be an order compatible equivalence relation on the set $\C{I}nt(\C{P})$ of intervals of $\C{P}$. Then the free $k$-module ${\bf C}(\C{P},\sim)$ generated by the set $\C{I}nt(\C{P})/\sim$ of types is a coalgebra, see \cite[Thm. 3.1]{Schm94}, with the comultiplication
\begin{equation*}
\D:{\bf C}(\C{P},\sim) \lra {\bf C}(\C{P},\sim)\ot {\bf C}(\C{P},\sim), \qquad \D(\wbar{[x,y]}) = \sum_{z \in [x,y]} \wbar{[x,z]} \ot \wbar{[z,y]},
\end{equation*}
and the counit
\begin{equation*}
\ve:{\bf C}(\C{P},\sim) \lra k, \qquad \ve(\wbar{[x,y]}) = \begin{cases}
1, \text{ if } x = y,\\
0, \text{ otherwise}.
\end{cases}
\end{equation*}
Following the terminology of \cite{Schm87}, the coalgebra ${\bf C}(\C{P},\sim)$ is called the {\it reduced incidence coalgebra} of $\C{P}$ modulo the relation $\sim$. If $\sim$ is the interval isomorphism, then ${\bf C}(\C{P},\sim)$ is called the {\it standard reduced incidence coalgebra of $\C{P}$}. If, on the other extreme, $\sim$ is the trivial equivalence relation, then this incidence coalgebra is denoted simply by ${\bf C}(\C{P})$, and is called the {\it full (or unreduced) incidence coalgebra} associated to $\C{P}$.



\subsection{The divided powers coalgebra $\Div$}\label{DividedPowers} 
Let $\B{Z}^{\geq 0} = \{0,1,2,\ldots \}$ be the set of non-negative integers with the usual ordering, and $\sim$ the isomorphism of intervals. As we have observed earlier in Example \ref{ex1}, the type $\wbar{[i,j]}$ of a segment $[i,j]$ is determined by the nonnegative number $j-i \geq 0$. As a result, we obtain the standard reduced incidence coalgebra 
\begin{equation*}
{\bf C}(\B{Z}^{\geq 0},\sim) \cong \langle x_n \mid n\geq0\rangle
\end{equation*}
with the coalgebra structure given by
\begin{equation*}
\D(x_n) = \sum_{r=0}^n x_r \ot x_{n-r}, \qquad \ve(x_n) = \d_{n,0}.
\end{equation*}
We note that ${\bf C}(\B{Z}^{\geq 0},\sim)$ is the divided powers coalgebra, \cite[Chpt. XII]{Sweedler-book}, which we denote by $\Div$. Its convolution algebra ${\bf C}(\B{Z}^{\geq 0},\sim)^\ast$ is the reduced incidence algebra ${\bf R}(\B{Z}^{\geq 0},\sim)$ of Example \ref{ex1}.

\subsection{The binomial coalgebra $\Bin$}\label{BinomialCoalgebra}
Let $B(S)$ be the family of all finite subsets of a countable set $S$, ordered by inclusion. Let, as earlier, $\sim$ be the interval isomorphism. Then the type $\wbar{[A,B]}$ is determined by the number $|B\backslash A|$ of the elements of the set $B\backslash A$. As a result, we obtain the reduced incidence coalgebra 
\begin{equation*}
{\bf C}(B(S),\sim) \cong \langle y_n  \mid n\geq 0 \rangle
\end{equation*}
with the coalgebra structure given by
\begin{equation*}
\D(y_n) = \sum_{r=0}^n \binom{n}{r} y_r \ot y_{n-r}, \qquad \ve(y_n) = \d_{n,0}.
\end{equation*}
This is the binomial coalgebra, \cite[Def. 3.2.18]{SpiegelODonnell-book} which we denote by $\Bin$ in the sequel. The convolution algebra ${\bf C}(B(S),\sim)^\ast$ of this reduced incidence coalgebra is the reduced incidence algebra ${\bf R}(B(S),\sim)$ of Example \ref{ex2}.

\subsection{The quantum binomial coalgebra $\Bin_q$}\label{QuantumBinomialCoalgebra}
For $L(V)$ and $\sim$ as in Example \ref{ex4} we obtain the Eulerian coalgebra ${\bf C}(L(V),\sim) \cong \B{Z}^{\geq 0}$ of \cite[Def. 3.2.28]{SpiegelODonnell-book} given by
\begin{equation*}
\D(y_n) = \sum_{r=0}^n \frac{(q^n-1)(q^n-q)\ldots (q^n-q^{r-1})}{(q^r-1)(q^r-q)\ldots (q^r-q^{r-1})} y_r \ot y_{n-r}, \qquad \ve(y_n) = \d_{n,0},
\end{equation*}
which can be written more concisely as
\begin{equation*}
  \D(y_n) = \sum_{r=0}^n {n\brack k}_q y_r\otimes y_{n-r}
\end{equation*}
for every $n\geq 0$.  We denote this coalgebra by $\Bin_q$. Finally, the convolution algebra ${\bf C}(L(V),\sim)^\ast$ is the reduced incidence algebra ${\bf R}(L(V),\sim)$ of Example \ref{ex4}.

\subsection{The Dirichlet coalgebra $\Dir$}\label{DirichletCoalgebra}
Let $L(\B{Q})$ and $\sim$ be as in Example \ref{ex3}. Then we have the reduced incidence coalgebra ${\bf C}(L(\B{Q}),\sim)\cong\langle z_n \mid n\in \B{N} \rangle$ which is given by
\begin{equation*}
\D(z_n) = \sum_{ij=n} z_i \ot z_j, \qquad \ve(z_n) = \d_{n,1},
\end{equation*}
called the Dirichlet coalgebra in \cite[Def. 3.2.25]{SpiegelODonnell-book}. We denote the Dirichlet coalgebra by $\Dir$ in the sequel. It is evident that the convolution algebra ${\bf C}(L(\B{Q}),\sim)^\ast$ is isomorphic to the reduced incidence algebra ${\bf R}(L(\B{Q}),\sim)$ of Example \ref{ex3}.

\subsection{The graph incidence coalgebra}\label{graph-coalg-comm}
Let $\C{P}(\C{G})$ and $\sim$ be as in Subsection \ref{graph-incidence}. We thus have the reduced incidence coalgebra ${\bf C}(\C{P}(\C{G}),\sim)$ of isomorphism classes of graphs in $\C{G}$, with the coproduct
\begin{equation*}
\D([G]) = \sum_{U\subseteq V} [G\mid U] \ot [G\mid (V-U)]
\end{equation*} 
for any $G=(V,E)\in \C{G}$. Hence, the convolution algebra ${\bf C}(\C{P}(\C{G}),\sim)^\ast$ is isomorphic with the commutative graph incidence algebra ${\bf I}(\C{P}(\C{G}),\sim)$ of Subsection \ref{graph-incidence}. 

It follows from \cite[Thm. 10.2]{Schm94} that ${\bf C}(\C{P}(\C{G}),\sim) \cong k[\lambda(\widetilde{\C{G}_0})]$ as coalgebras, where
\begin{equation*}
\lambda: {\bf C}(\C{P}(\C{G}),\sim) \lra P({\bf C}(\C{P}(\C{G}),\sim)), \qquad \lambda([G]):= \sum_\pi(-1)^{|\pi|-1}(|\pi|-1)!\prod_{B\in \pi}[G\mid B],
\end{equation*}
is the projection onto the set of primitive elements, and the sum is over the set of all partitions $\pi$ of the vertex set $V$.

\begin{remark}
Following the terminology of Subsection \ref{graph-incidence} and Subsection \ref{graph-coalg-comm}, we note that if $\C{G}$ is the set of all (finite) graphs with no edges, then ${\bf C}(\C{P}(\C{G}),\sim) \cong \Bin$, see \cite[Ex. 12.2]{Schm94}. If, on the other hand, $\C{G}$ is the disjoint unions of complete graphs, then  ${\bf C}(\C{P}(\C{G}),\sim) \cong \Div$, see \cite[Ex. 12.3]{Schm94}.
\end{remark}

\section{Isomorphisms between the fundamental examples}\label{Isomorphisms}

\subsection{$\Div$, $\Bin$ and $\Bin_q$}

In this subsection we show that the (reduced incidence) coalgebras $\Div$, $\Bin$ and $\Bin_q$ are isomorphic as coalgebras. To this end, we observe that the mapping
\begin{equation}\label{theta-Bin}
\t:\Bin \lra k[X],\qquad y_n\mapsto X^n
\end{equation}
is a coalgebra isomorphism. Indeed, for every $n\geq 0$ we have
\begin{equation*}
(\t\ot\t)  \Delta(y_n)  = \sum_{i=0}^{n}\binom{n}{i}\t(y_i)\otimes \t(y_{n-i})
  = \sum_{i=0}^{n}\binom{n}{i}X^i\otimes X^{n-i}= \D(X^n).
\end{equation*}
We thus conclude the following.

\begin{proposition}\label{BinPoly}
The binomial coalgebra $\Bin$ is isomorphic with the polynomial coalgebra $k[X]$ as coalgebras.
\end{proposition}

We proceed to show that $\Bin$, $\Bin_q$ and $\Div$ are the same, and hence they are all are isomorphic with $k[X]$, as coalgebras.

\begin{proposition}\label{DivBin}
The coalgebras $\Div$, $\Bin$ and $\Bin_q$ are isomorphic.
\end{proposition}

\begin{proof}
Let us first define
  \begin{equation}\label{gamma}
\gamma\colon \Div\to \Bin,\qquad    \gamma(x_n) = \frac{y_n}{n!},
  \end{equation}
for any $n\geq 0$. The fact that \eqref{gamma} is an isomorphism of
  $k$-vector spaces is clear. On the basis elements we have
  \begin{align*}
 \Delta\gamma(x_n) 
    = & \sum_{i=0}^n \frac{1}{n!}\binom{n}{i}y_i\otimes y_{n-i}
    =  \sum_{i=0}^n \frac{y_i}{i!}\otimes \frac{y_{n-i}}{(n-i)!}
    =  \sum_{i=0}^n \gamma(x_i)\otimes\gamma(x_{n-i})\\
    = & (\gamma\otimes\gamma)\Delta(x_n)
  \end{align*}
  proving that \eqref{gamma} is an isomorphism of coalgebras. Similarly 
    \begin{equation}\label{gamma}
\gamma_q\colon \Div\to \Bin_q,\qquad    \gamma_q(x_n) = \frac{y_n}{[n]_q!}
  \end{equation}
  establishes an isomorphism between $\Div$ and $\Bin_q$.
\end{proof}

We record the following proposition, which follows directly from the definition of a tree, to unify these examples.

\begin{proposition}\label{Fundamental1}
  A finite poset $(P,\leq)$ has the property that every interval
  $[x,y]$ of $(P,\leq)$ is isomorphic to the interval $[0,n] = \{0,\ldots,n\}$ in
  $(\B{Z}^{\geq0},\leq)$ for some $n \geq 0$ if and only if the Hasse diagram of $P$
  is a forest.
\end{proposition}

\begin{proof}
  ($\Rightarrow$) We will prove the contrapositive of the statement in
  this direction.  Let $P$ be a poset whose Hasse diagram contains
  a cycle, that is, let there be two elements $x,y\in P$ such that
  there are two paths from $x\in P$ to $y\in P$ in the Hasse diagram of $P$.
  Then the interval $[x,y]$ can not be isomorphic to an interval of
  the form $\{0,\ldots,n\}\subseteq\B{Z}^{\geq0}$ with its canonical order
  structure.

  ($\Leftarrow$) Let the Hasse diagram of $P$ be a forest. Then, given
  two comparable elements $x\leq y$, there is a unique path from $x\in P$
  to $y\in P$ in the Hasse diagram.  Hence, the interval $[x,y]$ is
  isomorphic to $\{0,\ldots,n\}\subseteq\B{Z}^{\geq0}$, where $n\geq0$ is the number of
  elements in $[x,y]$.
\end{proof}

The \emph{depth} of a forest is the supremum of lengths of all paths
in that forest.  As a result, we identify the coalgebras $\Div$,
$\Bin$ and $\Bin_q$ with the reduced incidence coalgebra of a poset
whose Hasse diagram is a forest whose depth is not finite.

\begin{proposition}\label{Fundamental2}
  Let $(P_{\rm forest},\leq)$ be a locally finite poset whose Hasse diagram is a forest whose depth is not finite, and let $\sim$ be the
interval isomorphism. Then the reduced incidence coalgebra ${\bf C}(P_{\rm forest},\sim)$ is 
isomorphic to the coalgebras $\Div$, $\Bin$ and $\Bin_q$.
\end{proposition}

The coalgebra $\Dir$ is slightly different. However, similar to \eqref{theta-Bin}, we have an isomorphism
\begin{equation*}
\eta: \Dir \lra k[\{X_p\mid p\,\,\text{prime}\}],\qquad z_n\mapsto X_{p_1}\ldots X_{p_{s(n)}},
\end{equation*} 
of coalgebras, where $n=p_1\ldots p_{s(n)}$ is the prime factorization of $n\in \B{N}$.

\section{Cohomological Background}

\subsection{Co(algebra) Hochschild cohomology}
Referring the reader to \cite{AbraWeib02,BrzeWisb-book,Doi81,KaygKhal06,Kayg12,KaygSutl14} for the basics of the (co)coalgebra Hochschild cohomology and the $\cotor$-groups, we recall in this subsection the computational machinery introduced in \cite{KaygSutl14}.

Let $\C{C}\to \C{D}$ be a coflat coalgebra coextension, and $\C{Z}:= \C{C}\oplus \C{D}$ an auxiliary coalgebra with the comultiplication
\begin{equation*}
\Delta(y) = y_{(1)}\otimes y_{(2)} \quad\text{ and }\quad
   \Delta(x) = x_{(1)}\otimes x_{(2)} + \pi(x_{(1)})\otimes x_{(2)} + x_{(1)}\otimes \pi(x_{(2)})
\end{equation*}
and the counit
\begin{equation*}
\ve(x+y) = \ve(y),
\end{equation*}
for any $x\in \C{C}$ and $y\in \C{D}$.

Let also $V,W$ be $\C{C}$-comodules with opposite pairing. We note that in this case a right (resp. left) $\C{C}$-comodule $V$, by $v\mapsto v\ns{0}\ot v\ns{1}$ (resp. $v\mapsto v\ns{-1}\ot v\ns{0}$) is a right (resp. left) $\C{Z}$-comodule via $v\mapsto v\ns{0}\ot (v\ns{1},\pi(v\ns{1}))$ (resp. $v\mapsto (v\ns{-1},\pi(v\ns{-1})\ot v\ns{0})$).

In this case, the $\cotor$-groups of $\C{Z}$ are isomorphic with the $\cotor$-groups of $\C{C}$, which we record below.

\begin{lemma}
Given a coextension $\pi:\C{C}\to \C{D}$ of coalgebras, let $V,W$ be two $\C{C}$-comodules of opposite parity. Then the short exact sequence
\begin{equation*}
  0 \to \C{D} \xrightarrow{\ i\ } \C{Z} \xrightarrow{\ p\ } \C{C} \to 0
  \qquad \text{ where }
  \qquad i: y\mapsto (0,y)
  \qquad p: (x,y)\mapsto x
\end{equation*}
induces an isomorphism
\begin{equation*}
\cotor^n_\C{Z}(V,W)\cong \cotor^n_\C{C}(V,W), \qquad n\geq 0.
\end{equation*}
\end{lemma}

\begin{proof}
The proof follows from the proof of \cite[Lemma 4.10]{FariSolo00} taking the coefficients to be $W \ot V$.
\end{proof}

Now, let $\C{C}$ be coflat both as a
left and a right $\C{D}$-comodule, \ie $\cotor^n_\C{D}(V,\C{C}) = 0$ for $n\geq 1$, and for any right $\C{D}$-comodule $V$, and $\cotor^n_\C{D}(\C{C},W) = 0$ for $n\geq 1$, and for any left $\C{D}$-comodule $W$. Then consider $\cotor^\ast(V,\C{Z},W)$ with the decreasing filtration
\begin{equation*}
F^{n+p}_p =  \left\{ \begin{array}{ll}
          \bigoplus_{n_0+\cdots+n_p=n} V\otimes \C{Z}^{\otimes n_0}\otimes \C{C}\otimes\cdots\otimes \C{Z}^{n_{p-1}}\otimes \C{C}\otimes \C{Z}^{\otimes n_p}\ot W, & p\geq 0 \\
          0, & p<0.
        \end{array}
\right.
\end{equation*}
Then the spectral sequence associated to this filtration is
\begin{equation*}
E^{i,j}_0 = F^{i+j}_i / F^{i+j}_{i+1}
   = \bigoplus_{n_0+\cdots+n_i=j} V\otimes \C{D}^{\otimes n_0}\otimes \C{C}\otimes\cdots\otimes \C{D}^{n_{i-1}}\otimes \C{C}\otimes \C{D}^{\otimes n_i}\ot W,
\end{equation*}
and by the coflatness assumption, on the vertical direction it computes
\begin{equation*}
E_1^{i,j}=\cotor^j_\C{D}(V,\C{C}^{\Box_\C{D}\,i}\,\Box_\C{D}\,W).
\end{equation*}

As a result, we obtain the following analogue of \cite[Thm. 3.1]{KaygSutl14}.

\begin{theorem}\label{thm-spec-seq}
Let $\pi:\C{C}\lra \C{D}$ be a coalgebra coextension and $V,W$ be two $\C{C}$-comodules of opposite parity. Let also $\C{C}$ be coflat both as a left and a right $\C{D}$-comodule. Then there is a spectral sequence whose $E_1$-term is of the form
\begin{equation*}
E_1^{i,j}=\cotor^j_\C{D}(V,\C{C}^{\Box_\C{D}\,i}\,\Box_\C{D}\,W),
\end{equation*}
converging to $\cotor^{i+j}_\C{C}(V,W)$.
\end{theorem}

\subsection{Cohomology of Profinite Algebras}

In this subsection we recall, from \cite[Sect. 4]{AbraWeib02}, the identification of the cotor groups of a coalgebra with the Hochschild cohomology of its profinite dual. In the finite dimensional case, one has the following.

\begin{theorem}\label{finite-cotor}\cite[Thm. 3.4]{AbraWeib02}
Let $\C{A}$ be a finite dimensional algebra, and $\C{C}:=\C{A}^\ast$ its dual coalgebra. If $M$ is a left $\C{A}$-module, and $N$ a right $\C{A}$-module, then 
\begin{equation*}
\cotor_{\C{C}}^\ast(M,N) = HH^\ast(\C{A},M\ot N).
\end{equation*}
\end{theorem}

For an infinite dimensional analogue, we recall the following terminology. A profinite algebra is the inverse limit (projective limit) of finite dimensional algebras. Therefore, we first note that given a coalgebra $\C{C}$, the convolution algebra $\C{A} := \C{C}^\ast$ is a profinite algebra. 

Let $\C{C}$ be a coalgebra, and $\C{A} = \C{C}^\ast$ the dual (convolution) algebra. Then any left $\C{C}$-comodule $N$ is a right $\C{A}$-module via
\begin{equation*}
n\cdot a := \langle n\ns{-1},a \rangle n\ns{0},
\end{equation*}
and similarly any right $\C{C}$-comodule $M$ is a left $\C{A}$-module via
\begin{equation*}
a\cdot m := m\ns{0}\langle m\ns{1},a \rangle.
\end{equation*}

There is then the following infinite dimensional counterpart of Theorem \ref{finite-cotor}.

\begin{theorem}\cite[Thm. 4.11]{AbraWeib02}\label{thm-AbraWeib}
Let $\C{C}$ be a coalgebra, and $\C{A}:=\C{C}^\ast$ its profinite dual. If $M$ is a right $C$-comodule, and $N$ a left $C$-comodule, then
\begin{equation*}
\cotor_\C{C}^\ast(M,N) = HH^\ast_{\rm cts}(\C{A},M\ot N).
\end{equation*}
\end{theorem}

\section{Calculations}

\subsection{Cohomology of the fundamental examples}\label{Calculations} In this subsection we compute the Hochschild cohomology of the incidence algebras of Subsect. \ref{ex1}, Subsect. \ref{ex2}, Subsect. \ref{ex4}, and Subsect. \ref{ex3}, with trivial coefficients. We recall that these incidence algebras are the profinite duals of the coalgebras of Subsect. \ref{DividedPowers}, Subsect. \ref{BinomialCoalgebra}, Subsect. \ref{QuantumBinomialCoalgebra}, and Subsect. \ref{DirichletCoalgebra}, respectively. In the following theorem ${\bf R}(P_{\rm forest},\sim)$ denotes the convolution algebra ${\bf C}(P_{\rm forest},\sim)^\ast$.

\begin{theorem}\label{thm-exp}
Let $\C{A}$ be one of the following reduced incidence algebras: ${\bf R}(B(S),\sim)$, ${\bf R}(\B{Z}^{\geq 0},\sim)$, ${\bf R}(L(V),\sim)$ or ${\bf R}(P_{\rm forest},\sim)$. Then the continuous Hochschild cohomology, with trivial coefficients, of $\C{A}$ is given by
\begin{equation*}
HH_{\rm cts}^n(\C{A},k) = \begin{cases}
k & \text{\rm if } n=0,\\
k & \text{\rm if } n=1,\\
0 & \text{\rm if } n\geq 2.
\end{cases}
\end{equation*}
\end{theorem}

\begin{proof}
Since the reduced incidence algebra ${\bf R}(B(S),\sim)$ is the profinite dual of the reduced incidence coalgebra $\Div$ of Sect. \ref{sect-IncCoalg}, we use \cite[Thm. 4.11]{AbraWeib02} (Theorem \ref{thm-AbraWeib} above) to get
\begin{equation*}
HH_{\rm cts}^\ast({\bf R}(B(S),\sim),k) = HH_{\rm cts}^\ast({\bf R}(B(S),\sim),k\ot k) \cong \cotor_{\Bin}^\ast(k,k) \cong \cotor_{k[X]}^\ast(k,k),
\end{equation*}
where the last isomorphism follows from Proposition \ref{BinPoly}. Then the claim follows from \cite[Thm. XVIII.7.1]{Kassel-book}, and Proposition \ref{DivBin}.
\end{proof}

As for the reduced incidence algebra ${\bf R}(L(\B{Q}),\sim)$ of formal Dirichlet power series we first note the following.

\begin{lemma}
The Dirichlet coalgebra is isomorphic with the polynomial coalgebra $k[\{X_n\mid n\in \B{N}\}]$ of countable indeterminates.
\end{lemma}

\begin{proof}
  For every prime $p$, let $o(p)$ be the place of the prime $p$ among
  all primes sorted in their natural order in $\B{N}$.  For instance,
  $o(2)=1$, $o(3)=2$, $o(17)=7$, etc. Then the claimed isomorphism is given by 
  $X_p\mapsto X_{o(p)}$.
\end{proof}

We now are going to calculate the Hochschild cohomology of the algebra of formal Dirichlet series, with trivial coefficients. Given any $m\in \B{N}$, we define the subcoalgebra $\Dir_{(m)} := k[\{X_n\mid n\leq m\}]$. We thus have the coextensions of the form
\begin{align}\label{CoExtension}
& \pi_m\colon \Dir_{(m+1)}\to \Dir_{(m)}, \qquad \pi_n(X_1^{a_1}\cdots X_{n+1}^{a_{n+1}}) = 
  \begin{cases}
    X_1^{a_1}\cdots X_n^{a_n} & \text{ if } a_{n+1}=0,\\
    0 & \text{ otherwise.}
  \end{cases}
\end{align}

We remark that the natural embedding $\Dir_{(m)}\to \Dir_{(m+1)}$ splits \eqref{CoExtension}. The following result then follows immediately.

\begin{lemma}\label{lemma-coext}
For every $m\geq 1$, the coalgebra $\Dir_{(m+1)}$ is isomorphic with $\Dir_{(m)}\otimes k[X_{m+1}]$
(resp. $k[X_{m+1}]\otimes \Dir_{(m)}$) as left (resp. right) $\Dir_{(m)}$-comodules, whose left (resp. right) 
coaction is given by the comultiplication on $\Dir_{(m)}$. In particular, the coextension \eqref{CoExtension} is coflat for any $m\geq 1$.
\end{lemma}

In the construction of the Hochschild-groups of the (reduced) incidence algebra ${\bf R}(L(\B{Q}),\sim)$, we will use the the graded algebra structure of the $\cotor$-groups of the (reduced) incidence coalgebra ${\bf C}(L(\B{Q}),\sim)$. We thus record the following result.

\begin{lemma}
Let $C$ be a coalgebra with a group-like
  element $1 \in C$.  Considering $k$ as a left and right $C$-comodule via this element, $\cotor^*_{C}(k,k)$ is a
  graded algebra.
\end{lemma}

\begin{proof}
Following the notation of \cite{KaygSutl14}, let $\Psi\in \CB^p(k,C,k)$, and $\Phi\in \CB^q(k,C,k)$.  The
  straight tensor product of these elements $\Psi\otimes\Phi \in \CB^{p+q}(k,C,k)$, and the
  differential interacts as
  \begin{equation*}
    d^{p+q}(\Psi\otimes\Phi) 
    = d^p(\Psi)\otimes\Phi + (-1)^{p} \Psi\otimes d^q(\Phi).  
  \end{equation*}
  The result follows.
\end{proof}

Dually, there is a graded algebra structure on the Hochschild groups of an augmented algebra given by the Gerstenhaber (cup) product, see \cite{Gers63},
\begin{equation*}
f\cup g (a_1,\ldots, a_{p+q}) := f(a_1,\ldots, a_p)g(a_{p+1},\ldots, a_{p+q}).
\end{equation*} 

We can now state, and prove, the following infinite dimensional analogue of Theorem \ref{thm-exp}.

\begin{theorem}\label{Fundamental3}
The Hochschild cohomology, with trivial coefficients, of the reduced incidence algebra ${\bf R}(L(\B{Q}),\sim)$ of formal Dirichlet series is a graded commutative algebra with countable generators.
\end{theorem}

\begin{proof}
In view of Lemma \ref{lemma-coext} and Theorem \ref{thm-spec-seq}, for each $m\geq 1$ we have a spectral sequence
  \begin{equation*}
    E^{p,q}_1 = \cotor^q_{\Dir_{(m)}}(k, {\Dir_{(m+1)}}^{\Box_{\Dir_{(m)}}\,p}\,\Box_{\Dir_{(m)}}\,k) \Rightarrow \cotor^{p+q}_{\Dir_{(m+1)}}(k,k).
  \end{equation*}
Since
  \begin{equation*}
\Dir_{(m+1)}^{\Box_{\Dir_{(m)}}\,p}\,\Box_{\Dir_{(m)}}\,k \cong k[X_{m+1}]^{\ot\,p} \ot k
  \end{equation*}
  as left $\Dir_{(m)}$-comodules, 
  \begin{equation*}
E_1^{p,q} \cong \cotor^q_{\Dir_{(m)}}(k,k[X_{m+1}]^{\ot\,p} \ot k) = \cotor^q_{\Dir_{(m)}}(k,k)\ot k[X_{m+1}]^{\ot\,p} \ot k
  \end{equation*}  
As a result,
  \begin{equation*}
E_2^{p,q} \cong \cotor^q_{\Dir_{(m)}}(k,k)\ot \cotor^p_{k[X_{m+1}]}(k,k),
  \end{equation*}  
where the latter is the graded algebra generated by $[X_0]:=[1] \in E_2^{0,q}$, and $[X_{m+1}]\in E_2^{1,q}$. The claim now follows from the observation that the identification $\cotor^\ast_{{\bf C}(L(\B{Q}),\sim)}(k,k) \cong HH^\ast({\bf R}(L(\B{Q}),\sim),k)$ by Theorem \ref{thm-AbraWeib} is an isomorphism of differential graded algebras.
\end{proof}

We can immediately conclude the cohomology of commutative graph incidence algebras.

\begin{corollary}
Let $\C{G}$ be the set of all (finite) simple graphs, and $\C{P}(\C{G})$ the family of finite products of intervals of the lattice of subsets of the graphs in $\C{G}$, ordered by inclusion. Let also $\sim$ be the relation given by \eqref{sim}. Then the Hochschild cohomology $HH^\ast({\bf I}(\C{P}(\C{G}),\sim),k)$ of the commutative (reduced) graph incidence algebra ${\bf I}(\C{P}(\C{G}),\sim)$ is a graded commutative algebra with $|\lambda(\widetilde{\C{G}_0})|$-many generators.
\end{corollary}

\subsection{Cohomology via ${\rm Ext}$-groups}\label{Suspension}
Following \cite{Cibi89}, we let $\wbar{Q}$ be the extension of the quiver $Q$ adding two vertices $a$ and $b$, a new arrow from $b$ to each source of $Q$, and a new arrow from each sink of $Q$ to $a$. Let also $\wbar{I}$ be the parallel ideal of $k\wbar{Q}$, and $\wbar{\C{A}}:=k\wbar{Q}/\wbar{I}$ the incidence algebra associated to the quiver $\wbar{Q}$. 

Let $\C{A}^+\subseteq \wbar{\C{A}}$ be the subalgebra
\begin{equation*}
\C{A}^+ := \Big\langle \{\wbar{e_a},\wbar{e_b},\wbar{p}\mid p\in kQ\} \Big\rangle.
\end{equation*}

For each vertex $x\in \wbar{Q}$, let ${}_xk$ be the one dimensional left $\wbar{\C{A}}$-module given by the identity action of $\wbar{e_x}\in \wbar{\C{A}}$, and the trivial action of any other element of $\wbar{\C{A}}$. We similarly define the right $\wbar{\C{A}}$-module $k_x$. 

Let ${\C{C}^+}:=({\C{A}^+})^\ast$, and $\wbar{\C{C}}:=(\wbar{\C{A}})^\ast$ be the dual coalgebras. Dualizing the (flat) extension ${\C{A}^+}\to \wbar{\C{A}}$  we arrive at the coflat coextension $\wbar{\C{C}} \to {\C{C}^+}$ of finite dimensional coalgebras. 

We recall from \cite[Prop. 2.1]{Cibi89} and \cite[Coroll. 1.4]{Cibi89} (which is by \cite{GersScha83}) that the Hochschild cohomology groups of an incidence algebra can be related to the ${\rm Ext}$-groups. We provide in this subsection an alternative proof to this result. 

\begin{theorem}
Let $Q$ be an ordered quiver, and $\C{A}:=kQ/I$ the incidence algebra associated to it. Let $\wbar{\C{A}}:=k\wbar{Q}/\wbar{I}$ be the incidence algebra associated to the quiver $\wbar{Q}$ adjoining $Q$ two new vertices $a$ and $b$. Then,
\begin{equation*}
HH^n(\C{A},\C{A}) = {\rm Ext}_{\wbar{\C{A}}}^{n+2}(k_a,{}_bk), \qquad n\geq 1.
\end{equation*}
\end{theorem}

\begin{proof}
We conclude from \cite[Coroll. IX.4.4]{CartanEilenberg-book} and \cite[Thm. 3.4]{AbraWeib02} that
\begin{align*}
& {\rm Ext}_{\wbar{\C{A}}}^{n+2}(k_a,{}_bk) \cong HH^{n+2}({\wbar{\C{A}}},\Hom_k(k_a,{}_bk)) = \\
&\hspace{2cm} HH^{n+2}({\wbar{\C{A}}},{k_a}^\ast\ot{}_bk)) \cong \cotor^{n+2}_{\wbar{\C{C}}}({k_a}^\ast,{}_bk).
\end{align*}
We next employ Theorem \ref{thm-spec-seq} to obtain a spectral sequence
\begin{equation*}
E_1^{i,j}=\cotor^j_{{\C{C}^+}}({k_a}^\ast,\wbar{\C{C}}^{\Box^i_{{\C{C}^+}}}\,\Box_{{\C{C}^+}}\,{}_bk) \Rightarrow \cotor^{i+j}_{\wbar{\C{C}}}({k_a}^\ast,{}_bk).
\end{equation*}
It follows immediately from the very definition of ${\C{C}^+}$ that $E_1^{i,j}=0$ for $j\geq 1$, and hence we arrive at
\begin{equation*}
{\rm Ext}_{\wbar{\C{A}}}^{n+2}(k_a,{}_bk) \cong \cotor^{n+2}_{\wbar{\C{C}}}({k_a}^\ast,{}_bk) \cong E_2^{n+2,0} 
\end{equation*}
which is the homology
\begin{equation*}
E_2^{n+2,0} := H^{n+2}(\CB_\Box^\ast({k_a}^\ast,{\C{C}^+},{}_b k),d)
\end{equation*}
of the subcomplex 
\begin{equation*}
\CB_\Box^\ast(k_a^\ast,{\C{C}^+},{}_b k) := \bigoplus_{n\geq 0} \CB_\Box^\ast(k_a^\ast,{\C{C}^+},{}_b k), \qquad \CB_\Box^n({k_a}^\ast,{\C{C}^+},{}_b k) := {k_a}^\ast\,\Box_{{\C{C}^+}} \, \wbar{\C{C}}^{\Box^n_{{\C{C}^+}}} \, \Box_{{\C{C}^+}}\,{}_bk
\end{equation*}
of the complex $\CB^\ast({k_a}^\ast,\wbar{\C{C}},{}_bk)$.  We further note that
\begin{equation*}
E_2^{n+2,0} := H^{n+2}(\CB_\Box^\ast({k_a}^\ast,{\C{C}^+},{}_b k),d) = \cotor^n_{\C{C}}({}_a V,V_b),
\end{equation*}
where ${}_a V = {k_a}^\ast\, \Box_{{\C{C}^+}}\,\wbar{\C{C}}$ with the right ${\C{C}^+}$-coaction given by $(\Id\ot\Id\ot \pi)(\Id\ot \wbar{\D})$, and similarly $V_b = \wbar{\C{C}}\, \Box_{{\C{C}^+}}\, {}_b k$ with the left ${\C{C}^+}$-coaction given by $(\pi \ot \Id\ot\Id)(\wbar{\D} \ot \Id)$. Using \cite[Thm. 3.4]{AbraWeib02} once again, we obtain
\begin{equation*}
{\rm Ext}_{\wbar{\C{A}}}^{n+2}(k_a,{}_bk) \cong HH^n(\C{A},{}_a V\ot V_b).
\end{equation*}
Finally, we consider the short exact sequence
\begin{equation*}
\xymatrix{
0\ar[r] & {}_a k_b \ar[r] & {}_a V\ot V_b \ar[r] & \C{A} \ar[r] & 0
}
\end{equation*}
of $\C{A}$-bimodules given by
\begin{equation*}
\eta:{}_a V\ot V_b \lra \C{A}, \qquad \wbar{p}\ot \wbar{q} \mapsto \wbar{s},
\end{equation*}
where $\wbar{s} \in\wbar{\C{A}}$ is the unique path joining $t(p)\in \wbar{Q}$ to $s(q)\in \wbar{Q}$, if it exists, and $\eta(\wbar{p}\ot \wbar{q}) = 0$ if the exists no such path in $Q$. Hence the kernel is determined by the unique path from $a\in \wbar{Q}$ to $b\in \wbar{Q}$. We note also that for any $\wbar{t} \in \C{A}$, dualizing the right $\C{C}$-coaction of ${}_a V$, 
\begin{equation*}
\wbar{t} \cdot \wbar{p} = \begin{cases}
\wbar{p_1} & \text{\rm if }\,\wbar{p} = \wbar{p_1}\wbar{t} \\
0 & \text{\rm otherwise.}
\end{cases}
\end{equation*}
Therefore,
\begin{equation*}
\eta(\wbar{t} \cdot \wbar{p}\ot \wbar{q}) =\wbar{t}\eta(\wbar{p}\ot \wbar{q}),
\end{equation*}
and similarly
\begin{equation*}
\eta(\wbar{p}\ot \wbar{q}\cdot \wbar{t}) =\eta(\wbar{p}\ot \wbar{q})\wbar{t}.
\end{equation*}
As a result, we have a long exact sequence 
\begin{equation*}
\xymatrix{
\ldots \ar[r] & HH^m(\C{A},{}_a k_b) \ar[r] & HH^m(\C{A},{}_a V\ot V_b) \ar[r] & HH^m(\C{A},\C{A}) \ar[r] & HH^{m+1}(\C{A},{}_a k_b)\ar[r] & \ldots
}
\end{equation*}
of Hochschild groups. Finally, since 
\begin{equation*}
HH^m(\C{A},{}_a k_b)=0,\qquad m\geq 1,
\end{equation*}
we obtain
\begin{equation*}
HH^m(\C{A},{}_a V\ot V_b) \cong HH^m(\C{A},\C{A}), \qquad m\geq 1
\end{equation*}
finishing the proof.
\end{proof}

\subsection*{Acknowledgement} The project is supported by D\"uzce University BAP (Bilimsel Ara\c{s}t{\i}rma Projesi) grant 2014.05.04.236, ``Artin Cebirlerinin Homolojik Boyutlar{\i}".

\bibliographystyle{plain}
\bibliography{references}{}

\end{document}